\documentclass[12pt]{amsart}

\usepackage{fullpage}
\usepackage{amsfonts,amscd}
\usepackage{amssymb}
\usepackage{url}
\usepackage{graphicx}
\usepackage{tikz}
\usepackage[rightcaption]{sidecap}
\sidecaptionvpos{figure}{c}

\usepackage[english]{babel}

\allowdisplaybreaks

\theoremstyle{plain}
\newtheorem{theorem}                {Theorem}      [section]
\newtheorem{proposition}  [theorem]  {Proposition}
\newtheorem{corollary}    [theorem]  {Corollary}
\newtheorem{lemma}        [theorem]  {Lemma}

\theoremstyle{definition}

\newtheorem{remark}       [theorem]  {Remark}

\numberwithin{equation}{section}

\def \R{{\mathbb R}}

\def \s {{\mathbb S}}
\def \h {{\mathbb H}}

\def \Hy{{\mathbb H}}

\def \hor{\mathcal H}

\usepackage{color}

\def \link {~}
\def \1 {\`}

\DeclareMathOperator{\grad}{grad}
\DeclareMathOperator{\trace}{trace}
\numberwithin{equation}{section}

\title[Polyharmonic surfaces in $3$-dimensional homogeneous spaces]{Polyharmonic surfaces in $3$-dimensional homogeneous spaces}

\author{S.~Montaldo}
\address{Universit\`a degli Studi di Cagliari\\
Dipartimento di Matematica e Informatica\\
Via Ospedale 72\\
09124 Cagliari, Italia}
\email{montaldo@unica.it}

\author{C.~Oniciuc}
\address{Faculty of Mathematics\\ ``Al.I. Cuza'' University of Iasi\\
Bd. Carol I no. 11 \\
700506 Iasi, ROMANIA}
\email{oniciucc@uaic.ro}

\author{A.~Ratto}
\address{Universit\`a degli Studi di Cagliari\\
Dipartimento di Matematica e Informatica\\
Via Ospedale 72\\
09124 Cagliari, Italia}
\email{rattoa@unica.it}

\usepackage{hyperref}
\begin{document}
\begin{abstract}
In the first part of this paper we shall classify proper triharmonic isoparametric surfaces in $3$-dimensional homogeneous spaces (Bianchi-Cartan-Vranceanu spaces, shortly BCV-spaces). We shall also prove that triharmonic Hopf cylinders are necessarily CMC. In the last section we shall determine a complete classification of CMC $r$-harmonic Hopf cylinders in BCV-spaces, $r \geq3$. This result ensures the existence, for suitable values of $r$, of an ample family of new examples of $r$-harmonic surfaces in BCV-spaces.
\end{abstract}



\thanks{The authors S.M. and A.R. are members of the Italian National Group G.N.S.A.G.A. of INdAM. The author C.O. was supported by a project funded by the Ministry of
Research and Innovation within Program 1 - Development of the national RD system, Subprogram 1.2 - Institutional Performance - RDI excellence funding projects, Contract no. 34PFE/19.10.2018.}

\maketitle

\section{Introduction}\label{Intro}
In order to introduce the geometrical setting of this paper we recall that
\textit{harmonic maps} are the critical points of the {\em energy functional}
\begin{equation}\label{energia}
E(\varphi)=\frac{1}{2}\int_{M}\,|d\varphi|^2\,dV \, ,
\end{equation}
where $\varphi:M\to N$ is a smooth map between two Riemannian
manifolds $(M,g)$ and $(N,h)$. A map $\varphi$ is harmonic if it is a solution of the Euler-Lagrange system of equations associated to \eqref{energia}, i.e.,
\begin{equation}\label{harmonicityequation}
  - d^* d \varphi =   {\trace} \, \nabla d \varphi =0 \, .
\end{equation}
The left member of \eqref{harmonicityequation} is a vector field along the map $\varphi$ or, equivalently, a section of the pull-back bundle $\varphi^{-1} TN$: it is called {\em tension field} and denoted $\tau (\varphi)$. In addition, we recall that if $\varphi$ is an \textit{isometric immersion}, then $\varphi$ is a harmonic map if and only if the immersion $\varphi$ defines a minimal submanifold of $N$ (see \cite{MR703510, MR1363513} for background). Let us denote $\nabla^M$, $\nabla^N$ and $\nabla^{\varphi}$ the induced connections on the bundles $TM$, $TN$ and $\varphi ^{-1}TN$ respectively. The \textit{rough Laplacian} on sections of $\varphi^{-1}  TN$, denoted $\overline{\Delta}$, is defined by
\begin{equation} \label{roughlaplacian}
    \overline{\Delta}=d^* d =-\sum_{i=1}^m\left(\nabla^{\varphi}_{e_i}
    \nabla^{\varphi}_{e_i}-\nabla^{\varphi}_
    {\nabla^M_{e_i}e_i}\right)\,,
\end{equation}
where $\{e_i\}_{i=1}^m$ is a local orthonormal frame field tangent to $M$. 

Now, in order to define the notion of an $r$-harmonic map, we consider the following family of functionals which represent a version of order $r$ of the classical energy \eqref{energia}. 

If $r=2s$, $s \geq 1$:
\begin{eqnarray}\label{2s-energia}
E_{2s}(\varphi)&=& \frac{1}{2} \int_M \, \langle \, \underbrace{(d^* d) \ldots (d^* d)}_{s\, {\rm times}}\varphi, \,\underbrace{(d^* d) \ldots (d^* d)}_{s\, {\rm times}}\varphi \, \rangle_{_N}\, \,dV \nonumber\\ 
&=& \frac{1}{2} \int_M \, \langle \,\overline{\Delta}^{s-1}\tau(\varphi), \,\overline{\Delta}^{s-1}\tau(\varphi)\,\rangle_{_N} \, \,dV\,.
\end{eqnarray}
In the case that $r=2s+1$, $s\geq 0$:
\begin{eqnarray}\label{2s+1-energia}
E_{2s+1}(\varphi)&=& \frac{1}{2} \int_M \, \langle\,d\underbrace{(d^* d) \ldots (d^* d)}_{s\, {\rm times}}\varphi, \,d\underbrace{(d^* d) \ldots (d^* d)}_{s\, {\rm times}}\varphi\,\rangle_{_N}\, \,dV\nonumber \\
&=& \frac{1}{2} \int_M \,\sum_{j=1}^m \langle\,\nabla^\varphi_{e_j}\, \overline{\Delta}^{s-1}\tau(\varphi), \,\nabla^\varphi_{e_j}\,\overline{\Delta}^{s-1}\tau(\varphi)\, \rangle_{_N} \, \,dV \,.
\end{eqnarray}
We say that a map $\varphi$ is \textit{$r$-harmonic} if, for all variations $\varphi_t$,
$$
\left .\frac{d}{dt} \, E_{r}(\varphi_t) \, \right |_{t=0}\,=\,0 \,\,.
$$
This condition is equivalent to the vanishing of the $r$-tension field $\tau_r(\varphi)$. We recall that the expressions which describe the $r$-tension field of a general map $\varphi:M \to N$ between two Riemannian manifolds were computed by Maeta (see \cite{MR2869168}) and are the following:
\begin{eqnarray}\label{2s-tension}
\tau_{2s}(\varphi)&=&\overline{\Delta}^{2s-1}\tau(\varphi)-R^N \left(\overline{\Delta}^{2s-2} \tau(\varphi), d \varphi (e_i)\right ) d \varphi (e_i) \nonumber\\ 
&&  - \sum_{\ell=1}^{s-1}\, \left \{R^N \left( \nabla^\varphi_{e_i}\,\overline{\Delta}^{s+\ell-2} \tau(\varphi), \overline{\Delta}^{s-\ell-1} \tau(\varphi)\right ) d \varphi (e_i)  \right .\\ \nonumber
&& \qquad \qquad  -\, \left . R^N \left( \overline{\Delta}^{s+\ell-2} \tau(\varphi),\nabla^\varphi_{e_i}\, \overline{\Delta}^{s-\ell-1} \tau(\varphi)\right ) d \varphi (e_i)  \right \} \,\, ,
\end{eqnarray}
where $\overline{\Delta}^{-1}=0$ and $\{e_i\}_{i=1}^m$ is a local orthonormal frame field tangent to $M$ (the sum over $i$ is not written but understood). Similarly,
\begin{eqnarray}\label{2s+1-tension}
\tau_{2s+1}(\varphi)&=&\overline{\Delta}^{2s}\tau(\varphi)-R^N \left(\overline{\Delta}^{2s-1} \tau(\varphi), d \varphi (e_i)\right ) d \varphi (e_i)\nonumber \\ 
&&  -\sum_{\ell=1}^{s-1}\, \left \{R^N \left( \nabla^\varphi_{e_i}\,\overline{\Delta}^{s+\ell-1} \tau(\varphi), \overline{\Delta}^{s-\ell-1} \tau(\varphi)\right ) d \varphi (e_i)  \right .\\ \nonumber
&& \qquad \qquad  -\, \left . R^N \left( \overline{\Delta}^{s+\ell-1} \tau(\varphi),\nabla^\varphi_{e_i}\, \overline{\Delta}^{s-\ell-1} \tau(\varphi)\right ) d \varphi (e_i)  \right \} \\ \nonumber
&& \,-\,R^N \Big( \nabla^\varphi_{e_i}\,\overline{\Delta}^{s-1} \tau(\varphi), \overline{\Delta}^{s-1} \tau(\varphi)\Big ) d \varphi (e_i)\,\,. 
\end{eqnarray}
If $r=1$, the functional \eqref{2s+1-energia} is just the energy. In the case that $r=2$, the functional \eqref{2s-energia} is called \textit{bienergy} and its critical points are the so-called \textit{biharmonic maps}. At present, a very ample literature on biharmonic maps is available and, again, we refer to \cite{MR4265170} and references therein for an introduction to this topic. 

More generally, the \textit{$r$-energy functionals} $E_r(\varphi)$ defined in \eqref{2s-energia}, \eqref{2s+1-energia} have been intensively studied (see \cite{MR4106647, MR2869168, MR3403738, MR3371364, MOR-Israel, MR3711937, MR3790367}, for instance).  
Inspection of the Euler-Lagrange equations for $E_r(\varphi)$ shows that a harmonic map is always $r$-harmonic for any $r\geq 2$. When the target manifold is nonflat, we use to call an $r$-harmonic map  {\it proper} if it is not harmonic (similarly, an $r$-harmonic submanifold, i.e., an $r$-harmonic isometric immersion, is {\it proper} if it is not minimal). As a general fact, when the ambient space has nonpositive sectional curvature there are several results which assert that, under suitable conditions, an $r$-harmonic submanifold is minimal (see \cite{MR3403738} and \cite{MR3371364}, for instance).

Things drastically change when the ambient space is positively curved. Let us denote by  $\s^{m+1}$ the sphere $\s^{m+1}(1)$ of radius $1$. Moreover, let $A$ be the shape operator of $M^m$ into $\s^{m+1}$ and ${\mathbf H}=f \eta$ the mean curvature vector field, where $\eta$ is the unit normal vector field and $f$ is the mean curvature function. Throughout the whole paper, when we write that $M^m$ is a CMC hypersurface we mean that $f$ is a constant which will be denoted by $\alpha$. 

In \cite{Ou-Pacific-2010} Ou derived the equation for biharmonic hypersurfaces in a generic Riemannian manifold. More precisely, he proved: 
\begin{theorem}\cite{Ou-Pacific-2010}\label{MTH}
Let $\varphi:M^{m}\to N^{m+1}$ be an isometric immersion
of codimension-one with mean curvature vector ${\mathbf H}=f \eta $. Then
$\varphi$ is biharmonic if and only if:
\begin{equation}\label{BHEq}
\begin{cases}
\Delta f+f |A|^{2}-f{\rm
Ric}^N(\eta,\eta)=0,\\
 2A\,({\rm grad}\,f) +m f {\rm grad}\, f
-2\, f \,({\rm Ric}^N\,(\eta))^{\top}=0,
\end{cases}
\end{equation}
where ${\rm Ric}^N : T_qN\longrightarrow T_qN$ denotes the Ricci
operator of the ambient space defined by $\langle {\rm Ric}^N\, (Z),
W\rangle={\rm Ric}^N (Z, W)$ and  $A$ is the shape operator of the
hypersurface with respect to the unit normal vector $\eta$. 
\end{theorem}
We point out that, contrary to \cite{Ou-Pacific-2010}, the sign convention for $\Delta$ in this paper is such that $\Delta f =-f''$ on $\R$. 
If the mean curvature $f$ is constant, say $f\equiv \alpha$, then the biharmonic equation reduces to
\begin{equation}\label{hn45}
\begin{cases}
-\alpha |A|^{2}+\alpha {\rm
Ric}^N(\eta,\eta)=0\\
 \, \alpha \,({\rm
Ric}^N\,(\eta))^{\top}=0\,,
\end{cases}
\end{equation}
from which we deduce that a non minimal CMC hypersurface $M^m$ is \textit{proper} biharmonic if and only if 
\begin{equation}\label{eq:bihar-cmc-hyper-general}
{\rm
Ric}^N(\eta)=|A|^2\eta \,.
\end{equation}
In the instance that $M^m$ is a hypersurface of $\s^{m+1}$ the biharmonic condition \eqref{eq:bihar-cmc-hyper-general} reduces to
\begin{equation}\label{eq:bihar-cmc-hyper-sn}
|A|^2-m=0 \,.
\end{equation}
As for the $r$-harmonic case,  condition \eqref{eq:bihar-cmc-hyper-sn} was generalized in \cite{MOR-Israel}:
\begin{theorem}\label{Th-existence-hypersurfaces-c>0}
Let $M^m$ be a non-minimal CMC hypersurface in $\s^{m+1}$ and assume that $|A|^2$ is constant. Then
$M^m$ is proper $r$-harmonic ($r \geq 3$) if and only if
\begin{equation}\label{r-harmonicity-condition-in-spheres}
|A|^4-m|A|^2-(r-2)m^2 \alpha^2=0 \,.
\end{equation}
\end{theorem} 
As an application of Theorem\link\ref{Th-existence-hypersurfaces-c>0}, several new examples of isoparametric $r$-harmonic hypersurfaces were illustrated in \cite{MOR-Israel}, where it was stressed that the value of $r$, $r\geq 2$, plays a crucial role when the ambient is positively curved. By contrast, when the target space form has nonpositive curvature, generally non-existence results are confirmed for all values of $r$, $r \geq 2$.
\vspace{1mm}

As a natural further step, in this paper we shall focus on the study of $r$-harmonic surfaces into $3$-dimensional homogeneous spaces with group of isometries of dimension $4$. 

It is well-known (see, e.g., \cite{Belkhelfa-altri-book-2000}, \cite{Caddeo-altri-Mediterr-2006}, \cite{Daniel-Comm-Math-Helv-2007}) that $3$-dimensional homogeneous spaces with group of isometries of dimension $4$ admit, as a canonical model, the so called  Bianchi-Cartan-Vranceanu  spaces (shortly, BCV-spaces)
\begin{equation}\label{CV}
M^3_{m,\ell}=\left(\bar{M}\times \R,g=\frac{dx^2+dy^2}{[1+m(x^2+y^2)]^2}+\Big[dz+\frac{\ell}{2}\,\frac{y
dx-x dy}{1+m(x^2+y^2)}\Big ]^2\right)\,,
\end{equation}
where $\bar{M}=\{(x,y)\in\R^2\colon 1+m(x^2+y^2)>0\}$.

The space $M^3_{m,\ell}$ is the total space of the following Riemannian submersion over a simply connected complete surface $M^2(4m)$ of constant curvature $4m$, see \cite{Daniel-Comm-Math-Helv-2007}: 
\begin{eqnarray}\label{RSM}
\pi:
M^3_{m,\ell}\longrightarrow  M^2(4m)=\left(\bar{M},h=\frac{dx^2+dy^2}{[1+m(x^2+y^2)]^2}\right),\;\;\;
\pi(x,y,z)=(x, y).
\end{eqnarray}
We point out that while in $\s^n(\rho)$ the letter $\rho$ indicates the radius, in $M^2(4m),\,\h^2(4m)$ the real number within the brackets represents the sectional curvature.

These BCV-spaces are also a model for Thurston's eight $3$-dimensional geometries with the exception of the
hyperbolic space $\Hy^3$ and $\rm Sol$. More precisely, they include the $3$-dimensional space forms $ \R ^3$ ($\ell=m=0$), $ \s^3(1/\sqrt{m})$ ($\ell^2=4m$), the product spaces $M^2(4m)\times \R$ ($\ell=0$), ${\rm Nil_3}$ ($m=0$), $\widetilde{SL}(2,\R)$ ($\ell\neq 0,\, m< 0$) and $SU(2)$ ($\ell\neq 0,\, m>0$, $\ell^2\neq 4m$). See Figure~\ref{diagram-BCVspaces} for a representation of the BCV-spaces with respect to the values of the parameters $\ell$ and $m$.

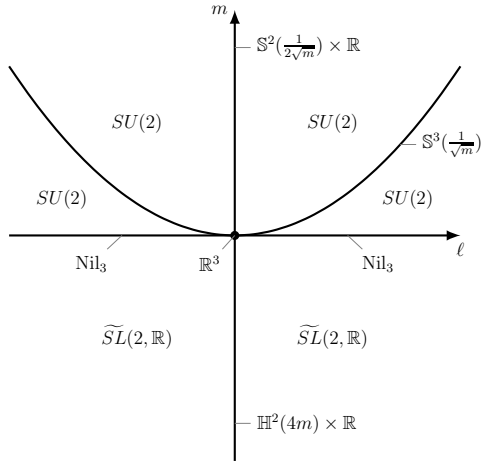
\begin{SCfigure}[50][h]
\label{diagram-BCVspaces}
\caption{Distribution of BCV-spaces w.r.t. to the values of $\ell$ and $m$.} 
\begin{tikzpicture}
\draw[thick, -latex] (-3,0) -- (3,0) node[below,scale=0.6] {$\ell$};
\draw[thick, -latex] (0,-3) -- (0,3) node[left,scale=0.6] {$m$};
\draw[scale=1, domain=-3:3, smooth, variable=\x,  thick] plot ({\x}, {0.25*\x*\x});
 \coordinate[pin={[pin distance=10,scale=0.6]300:${\rm Nil_3}$}] (r1) at (1.5,0);
  \coordinate[pin={[pin distance=10,scale=0.6]240:${\rm Nil_3}$}] (r1) at (-1.5,0);
  
\coordinate[pin={[pin distance=10,scale=0.6]0:$\s^3(\frac{1}{\sqrt{m}})$}] (r1) at (2.2,0.25*2.2*2.2);

\coordinate[pin={[pin distance=10,scale=0.6]0:$\s^2(\frac{1}{2\sqrt{m}})\times\R$}] (r1) at (0,2.5);
\coordinate[pin={[pin distance=10,scale=0.6]0:$\h^2(4m)\times\R$}] (r1) at (0,-2.5);

\node[below,scale=0.6] at (1.3,1.7) {$SU(2)$};
\node[below,scale=0.6] at (-1.3,1.7) {$SU(2)$};
\node[below,scale=0.6] at (2.3,0.7) {$SU(2)$};
\node[below,scale=0.6] at (-2.3,0.7) {$SU(2)$};
\filldraw [black] (0,0) circle (1.5pt);
 \coordinate[pin={[pin distance=10,scale=0.6]240:$\R^3$}] (r1) at (0,0);
\node[below,scale=0.6] at (1.3,-1.1) {$\widetilde{SL}(2,\R)$};
\node[below,scale=0.6] at (-1.3,-1.1) {$\widetilde{SL}(2,\R)$};
\end{tikzpicture}
\end{SCfigure}

In his paper \cite{Ou-J-Geom-Phys-2011}, Ou used equation (\ref{BHEq}) to study biharmonic surfaces in BCV-spaces. He first showed that
a totally umbilical biharmonic surface in any $3$-dimensional
Riemannian manifold has constant mean curvature. Then he used this to
show that the only totally umbilical proper biharmonic surface in
$3$-dimensional geometries is a part of $\s^2(1/\sqrt{2m})$ in $\s^3(1/\sqrt{m})$. 

Moreover, he proved the following characterization of CMC biharmonic surfaces:
\begin{theorem}\label{CVV}(see \cite{Ou-J-Geom-Phys-2011})
A CMC surface in a $3$-dimensional
 Bianchi-Cartan-Vranceanu space is proper biharmonic if and only if it is a
part of one of the following: 
\begin{enumerate}
\item[\rm (i)] $\s^2(\frac{1}{\sqrt{2m}})$ in $\s^3(\frac{1}{\sqrt{m}})$,
\item[\rm (ii)] $\s^1(\frac{1}{2\sqrt{2m}})\times \mathbb{R}$ in
$\s^2(\frac{1}{2\sqrt{2m}})\times \mathbb{R}$,
\item[\rm (iii)] a Hopf cylinder in $SU(2)$ with $4m-\ell^2>0$ over a circle
of radius $R=1/\sqrt{8m-\ell^2\:}$ in the base sphere
$M^2(4m)=\s^2(\frac{1}{2\sqrt{m}})$.
\end{enumerate}
\end{theorem}
\vspace{1mm}

The main aim of our paper is to investigate the existence of \textit{triharmonic} and, more generally, \textit{$r$-harmonic} surfaces in this geometric setting $(r \geq 3)$.

Our paper is organised as follows. In Section\link\ref{sec-results} we state our main results on triharmonic surfaces in BCV-spaces. These results will be proved in Sections\link\ref{sec-proofs} and \ref{sec-proofs2}. Finally, in Section\link\ref{sec-r-harmonic}, we shall determine a complete classification of proper CMC $r$-harmonic Hopf cylinders in Bianchi-Cartan-Vranceanu spaces, $r \geq3$. As an application, we shall be able to describe, for suitable values of $r$, an ample family of new examples of $r$-harmonic surfaces in BCV-spaces.

\section{Statement of the results on triharmonic surfaces in BCV-spaces}\label{sec-results}
In order to state our results, it is convenient to recall first some basic facts and terminology. 

For a Bianchi-Cartan-Vranceanu $3$-space given in (\ref{CV}), one can
easily check  that the vector fields
\begin{equation}\notag
E_{1}=F\frac{\partial}{\partial
x}-\frac{\ell y}{2}\frac{\partial}{\partial z},\quad E_{2}=F
\frac{\partial}{\partial y}+\frac{\ell x}{2}\frac{\partial}{\partial
z},\quad E_{3}=\frac{\partial}{\partial z},
\end{equation}
where $F=1+m(x^2+y^2)$, form a global orthonormal
frame field (see \cite{Caddeo-altri-Mediterr-2006, Ou-J-Geom-Phys-2011}).

Now, let $\gamma(s)=(x(s), y(s)),\,s \in I$ be a smooth curve in the base space $M^2(4m)$ of the Riemannian submersion \eqref{RSM}. Then the \textit{Hopf cylinder} $\Sigma_\gamma=\Sigma$ over the curve $\gamma$ is defined as
\begin{equation}\label{Hopf-cilindro}
\Sigma= \cup_{s\in I}\pi^{-1}(\gamma(s)) \,.
\end{equation}
Then the surface $\Sigma$ can
be parametrized as $r(s,t)=(x(s),y(s), t)$ since the fiber  of $\pi$
over a point $(x_0,y_0)$ is $\pi^{-1}(x_0,y_0)=\{(x_0,y_0, t): t\in
\R\}$.

It is convenient to assume that the base curve $\gamma$ is parametrized by arc length, i.e.,
\begin{equation*}\label{gamma-arc-length}
\frac{\dot{x}^2+\dot{y}^2}{F^2}= 1 \,.
\end{equation*}
Next, we define:
\begin{equation}\label{frame-field-adapted-Sigma}
X=\displaystyle{\frac{\dot{x}}{F}E_1+\frac{\dot{y}}{F}E_2}\,,\quad
\eta=\displaystyle{\frac{\dot{y}}{F}E_1-\frac{\dot{x}}{F}E_2} \,.
\end{equation}
Then the unit vector field $\eta$ is normal to the Hopf cylinder $\Sigma$ and $\{X,E_3,\eta \}$ is a global orthonormal frame field adapted to $\Sigma$.\\
 
Our first result is:
\begin{theorem}\label{Th-HopfcilindroCMC} Let $\Sigma$ be a triharmonic Hopf cylinder in a BCV-space $M^3_{m,\ell}$. Then $\Sigma$ is CMC.
\end{theorem}
The analysis in the proof of Theorem\link\ref{Th-HopfcilindroCMC} shows that, if $\Sigma$ is a Hopf cylinder, then its tension field is $\tau=-\kappa_g \eta$, where $\kappa_g$ denotes the geodesic curvature of its base curve.\\

Our second result is:
\begin{theorem}\label{Th-HopfcilindroCMC-espiliciti} Let $\Sigma$ be a CMC Hopf cylinder in a BCV-space $M^3_{m,\ell}$. 
\begin{itemize}
\item[{\rm (i)}]If $4m \leq \ell^2$ and $\Sigma$ is triharmonic, then $\Sigma$ is minimal.
\item[{\rm (ii)}]If $4m > \ell^2$ and the geodesic curvature $\kappa_g$ of its base curve verifies
\begin{equation}\label{kappag-CMCcilidro3armonico}
\kappa_g^2=2(4m-\ell^2) \,,
\end{equation}
then $\Sigma$ is proper triharmonic.
\end{itemize}
\end{theorem}

\begin{remark}
The base curve of a proper triharmonic Hopf cylinder in $SU(2)$ as in Theorem~\ref{Th-HopfcilindroCMC-espiliciti} (ii) is a circle of radius $R=1/\sqrt{12m-2\ell^2}$ in $\s^2(1/(2\sqrt{m}))$.
\end{remark}
The analysis of the Hopf cylinders fits naturally into the context of the study of isoparametric surfaces. We recall that, in a general Riemannian manifold, a hypersurface is said to be\textit{ isoparametric} if itself and its locally defined nearby equidistant hypersurfaces have constant mean curvature. In
the $30$'s, Cartan characterized isoparametric hypersurfaces in space forms as those with constant principal curvatures and achieved their classification in hyperbolic
spaces $\Hy^n$. Segre obtained a similar result for Euclidean spaces $\R^n$. In both cases, isoparametric hypersurfaces are also open parts of extrinsically homogeneous hypersurfaces, that is, codimension one orbits of isometric actions on the ambient
space. By contrast, the classification problem in spheres $\s^n$ is much more complicated and rich, and there are inhomogeneous examples (see \cite{MR4251144} and references therein, for instance). 

In spaces of nonconstant curvature, very few classification results are known. In the case of interest for us we have the following important result:
\begin{theorem}\label{Th-isop-BCV-spaces} \cite{MR4288655}
Let $\Sigma$ be an immersed surface in $M^3_{m,\ell}$, $4m-\ell^2 \neq 0$. Then the following assertions are equivalent:
\begin{itemize}
\item[{\rm (i)}] $\Sigma$ is an open subset of a homogeneous surface.
\item[{\rm (ii)}] $\Sigma$ is isoparametric.
\item[{\rm (iii)}] $\Sigma$ has constant principal curvatures.
\item[{\rm (iv)}] $\Sigma$ is an open subset of one of the following complete surfaces:

{\rm (a)} a Hopf cylinder over a complete curve of constant curvature in $M^2(4m)$;

{\rm (b)} a horizontal slice $M^2(4m) \times t_0$ with $\ell=0$;

{\rm (c)} a parabolic helicoid $P_{\alpha,m,\ell}$ with $\alpha^2+m<0$.
\end{itemize}
\end{theorem} 
\begin{remark} We point out that in references \cite{manzano-torralbo,Daniel-Comm-Math-Helv-2007,MR4288655} the authors use parameters $\kappa, \tau$ instead of $m,\ell$. The relationship between these parameters is $\kappa=4 m,\, \tau=\ell/2$. Also, in these papers our constant $\alpha$ is denoted by $H$. 
\end{remark}

The parabolic helicoids $P_{\alpha,m,\ell}$ will be described explicitly in Section\link\ref{sec-proofs}.

Our main result in the context of isoparametric surfaces is the following:
\begin{theorem}\label{Th-triharm-isop-BCV-spaces}
Let $\Sigma$ be an isoparametric immersed surface in $M^3_{m,\ell}$, $4m-\ell^2 \neq 0$. If $\Sigma$ is proper triharmonic, then it is an open part of a Hopf cylinder as in {\rm (ii)} of Theorem\link\ref{Th-HopfcilindroCMC-espiliciti}. 
\end{theorem}
\begin{remark} If $4m-\ell^2 = 0$, then $M^3_{m,\ell}$ is a space form with nonnegative sectional curvature and in this case the only proper triharmonic isoparametric surface is $\s^2(\frac{1}{\sqrt{3m}})$ in $\s^3(\frac{1}{\sqrt{m}})$ (see \cite{MR3403738}).
 \end{remark}
\section{Preliminaries}\label{sec-preliminaries}

In order to prepare the ground for our proofs we need to carry out some preliminary work. Generally, the use of a bar over a symbol indicates that we refer to an object of the ambient space. We adopt the following notation and sign convention for the Riemannian
curvature tensor field:
\begin{equation}\notag
 \overline{R}(X,Y)Z=\overline{\nabla}_{X}\overline{\nabla}_{Y}Z
-\overline{\nabla}_{Y}\overline{\nabla}_{X}Z-\overline{\nabla}_{[X,Y]}Z\,.
\end{equation}
Moreover,
\begin{equation}\notag
\begin{array}{lll}
&&  \overline{R}(X,Y,Z,W)=\langle \overline{R}(X,Y)W,Z \rangle ,\\
&&\\
\notag && {\rm \overline{Ric}}(X,Y)=\sum_{i=1}^3 \overline{R}(X, e_i, Y, e_i)=\sum_{i=1}^3 \langle \overline{R}(
X,e_i) e_i, Y\rangle.
\end{array}
\end{equation}

A straightforward computation shows that
\begin{equation}\label{Lie}
[E_1,E_2]=-[E_2,E_1]=2mxE_{2}-2myE_{1}+\ell E_{3}\,,\;\; {\rm all\;\;
others}\;\;[E_i,E_j]=0,\;\;i,j=1, 2, 3.
\end{equation}
Then, using the Koszul formula
\begin{eqnarray}\label{Koszul}
2\langle Z,\overline{\nabla}_Y X \rangle&=&X\langle Y,Z \rangle+Y\langle Z,X \rangle-Z\langle X,Y \rangle \\\nonumber
&&-\langle [X,Z],Y \rangle-\langle [Y,Z],X \rangle-\langle [X,Y],Z \rangle \,,
\end{eqnarray}
it is easy to compute:
\begin{equation}\label{CNil}
\begin{array}{ll}
\overline{\nabla}_{E_{1}}E_{1}=2myE_{2}\,,& \overline{\nabla}_{E_{2}}E_{2}=2mxE_{1}\,,\vspace{2mm} \\
\overline{\nabla}_{E_{1}}E_{2}=-2myE_{1}+\dfrac{\ell}{2}E_{3}\,,&
\overline{\nabla}_{E_{2}}E_{1}=-2mxE_{2}-\dfrac{\ell}{2}E_{3}\,,\vspace{2mm} \\
\overline{\nabla}_{E_{3}}E_{1}=\overline{\nabla}_{E_{1}}E_{3}=-\dfrac{\ell}{2}E_{2}\,,& \overline{\nabla}_{E_{3}}E_{2}=\overline{\nabla}_{E_{2}}E_{3}=\dfrac{\ell}{2}E_{1}\,,\vspace{2mm} \\
{\rm all \;\; others\;\;} \overline{\nabla}_{E_i}E_j=0,\;i,j=1, 2, 3.&
\end{array}
\end{equation}
Similarly, a further computation gives
the possible nonzero values of the sectional curvatures:
\begin{equation}\label{BCV1}
\begin{split}
 \overline{R}_{1212}&=\langle \overline{R}(E_{1},E_{2})E_{2},E_{1} \rangle=4m-\frac{3\ell ^2}{4},\\
\overline{R}_{1313}&=\langle \overline{R}(E_{1},E_{3})E_{3},E_{1}\rangle =\frac{\ell ^2}{4},\\
\overline{R}_{2323}&=\langle \overline{R}(E_{2},E_{3})E_{3},E_{2} \rangle=\frac{\ell ^2}{4}\,.
\end{split}
\end{equation}

The Riemannian curvature tensor field $\overline{R}$ of $M^3_{m,\ell}$ can be described as follows (see \cite{Daniel-Comm-Math-Helv-2007}, where the opposite sign convention for the curvature tensor is adopted).
\begin{eqnarray}\label{tensore-curvatura-general-expression}
 \overline{R}(X,Y)Z&=& \left ( 4m - \frac{3 \ell^2}{4}\right )\big( - \langle X,Z \rangle Y  + \langle Y,Z \rangle X \big )-(4 m - \ell^2)\nonumber \\
&& \big( \langle Y,E_3 \rangle \langle Z,E_3 \rangle X+
 \langle Y,Z \rangle \langle X,E_3 \rangle E_3-\langle X,E_3 \rangle \langle Z,E_3 \rangle Y-
 \langle X,Z \rangle \langle Y,E_3 \rangle E_3\big) \,.
\end{eqnarray}
Another useful formula is the following (see \cite{Daniel-Comm-Math-Helv-2007}):
\begin{equation}\label{nablaE3}
\overline{\nabla}_X \,E_3= \frac{\ell}{2} \, X \times E_3  \,,
\end{equation}
where $\times$ here has the following meaning:
\[
\langle X \times Y,Z \rangle =
\det{}_{(E_1,E_2,E_3)}(X,Y,Z) \, .
\]
We shall study oriented immersed surfaces $\varphi:M^2 \to M^3_{m,\ell}$ and denote by $\eta$ the unit normal vector field. 

The vector field $E_3^\top=E_3-\nu \, \eta$, where we have set
 \begin{equation}\label{nu-definition}
 \nu=\langle E_3,\eta \rangle \,,
 \end{equation}
will play a basic role in our analysis. This vector field plays an important part also in the previous literature on this subject. For our purposes, it is useful to recall (see Proposition 3.3 of \cite{Daniel-Comm-Math-Helv-2007}):
\begin{equation}\label{nablaE3T}
\overline{\nabla}_X \,E_3^\top=\nabla_X \,E_3^\top+B(X,E_3^\top)= \nu \, \left ( A(X)-\frac{\ell}{2} \,J( X)  \right ) + \langle A(X),E_3^\top \rangle \,\eta\,,
\end{equation}
where $J$ denotes the $\pi \slash 2$ rotation on $TM^2$. 

If $p$ is an arbitrarily fixed point of $M^2$, then, as $\nabla J=0$, we can consider a geodesic frame field $\{X_1,X_2\}$ such that, in a small neighbourhood of $p$, $J(X_1)=X_2$, $J(X_2)=-X_1$.

Moreover, taking into account the definition \eqref{nu-definition}, we compute:
\begin{equation}\label{nablaE3-eta}
X(\nu)= -\langle A(X)-\frac{\ell}{2} \,J( X)  , E_3 \rangle=-\langle  A(X)-\frac{\ell}{2} \,J( X)  , E_3^\top \rangle \,.
\end{equation}

\section{Proof of Theorems\link\ref{Th-HopfcilindroCMC} and \ref{Th-HopfcilindroCMC-espiliciti}}\label{sec-proofs}
Let $\{T,N\}$ denote the canonically oriented unit tangent and normal fields to $\gamma$ in the base space $M^2(4m)=(\bar{M},h)$, i.e.,
\[
T=(\dot{x},\dot{y}) \,, \quad N=(-\dot{y},\dot{x})\,,
\]
so that the geodesic curvature $\kappa_g$ of $\gamma$ is defined in the base space by means of
\[
\nabla_T T=\kappa_g N \,, \quad \nabla_T N=-\kappa_g T \,.
\]
Note that we denote $\overline{\nabla}=\nabla^g,\,\nabla=\nabla^\Sigma$. A computation shows that
\[
\kappa_g= \frac{2m}{F}(\dot{x}y-\dot{y}x) + \frac{\ddot{y}\dot{x}-\ddot{x}\dot{y}}{F^2} \,.
\]
Our first lemma is:
\begin{lemma}\label{lemma-nablaXX} Let $X,\eta$ be the vector fields defined in \eqref{frame-field-adapted-Sigma}. Then
\begin{equation}
\begin{array}{llllll}
{\rm (i)}& \overline{\nabla}_X X=-\kappa_g \eta\quad &{\rm (ii)}& \overline{\nabla}_X \eta=\kappa_g X -\dfrac{\ell}{2}E_3\quad&{\rm (iii)} &\overline{\nabla}_{E_3} X= \dfrac{\ell}{2} \eta\vspace{2mm} \\
{\rm (iv)}& \overline{\nabla}_{E_3} \eta=-\dfrac{\ell}{2}X\quad  &{\rm (v)}& \overline{\nabla}_{E_3} E_3=0&{\rm (vi)}& \overline{\nabla}_{X} E_3= \dfrac{\ell}{2} \eta
\end{array}
\end{equation}
\end{lemma}
\begin{proof}
We note that $X=T^{\hor}$, i.e., $X$ is the horizontal lift of $T$ and, similarly, $\eta=-N^{\hor}$. Because $\pi$ is a Riemannian submersion, we know that
\begin{equation}\label{Riem-submersion}
d\pi \left ( \overline{\nabla}_{V^\hor}W^\hor\right )= \nabla_{d\pi(V^\hor)}d\pi(W^\hor) \,.
\end{equation}
By way of example, we prove (ii): 
\[
d\pi \left( \overline{\nabla}_X \eta \right )=-\nabla_T N=\kappa_g T=d\pi (\kappa_g X)\,.
\]
It follows that
\[
\overline{\nabla}_X \eta=\kappa_g X + c E_3\,,
\]
where
\[
c= \langle \overline{\nabla}_X \eta, E_3 \rangle=-\langle  \eta, \overline{\nabla}_X E_3 \rangle \,.
\]
Next, using \eqref{nablaE3}, we deduce that
\[
c= - \frac{\ell}{2} \langle  \eta, X \times E_3 \rangle=- \frac{\ell}{2} \,.
\]
The other computations of this lemma are similar and so we omit them.
\end{proof}
It is easy to deduce from Lemma\link\ref{lemma-nablaXX} that the tension field of the Hopf cylinder is
\[
\tau = - \kappa_g \eta \,.
\]
In particular, the Hopf cylinder $\Sigma$ is CMC ($\alpha=-\kappa_g \slash 2$) if and only if its base curve $\gamma$ has constant geodesic curvature. 
Next, we need:
\begin{lemma}\label{lemma-Deltatau-cilindro}
\begin{equation}\label{Deltataucilindro}
\overline{\Delta} \tau =A X + B E_3 + C \eta \,,
\end{equation}
where
\[
A=3 \kappa_g \dot{\kappa}_g\,,\quad B=-\ell \dot{\kappa}_g\,,\quad C=
\ddot{\kappa}_g-\frac{\ell^2}{2} \kappa_g - \kappa_g^3 \,.
\]
\end{lemma}
\begin{proof}
We compute:
\[
\overline{\Delta} \tau =- \Big \{\overline{\nabla}_X \overline{\nabla}_X \tau-\overline{\nabla}_{\nabla_X X}\tau+ 
\overline{\nabla}_{E_3} \overline{\nabla}_{E_3} \tau-\overline{\nabla}_{\nabla_{E_3} E_3}\tau\Big \}\,.
\]
Since $\tau$ is orthogonal to $\Sigma$, we deduce from Lemma\link\ref{lemma-nablaXX}(i) that $\nabla_X X=0$. Also, $\nabla_{E_3} E_3=0$ and so, using again Lemma\link\ref{lemma-nablaXX}, we compute:
\begin{eqnarray*}
\overline{\Delta} \tau &=&- \Big \{\overline{\nabla}_X \overline{\nabla}_X \tau+ 
\overline{\nabla}_{E_3} \overline{\nabla}_{E_3} \tau\Big \}\\
&=&\overline{\nabla}_X \big [ \dot{\kappa}_g \eta+ \kappa_g(\kappa_g X-\frac{\ell}{2}E_3)\big ]-\overline{\nabla}_{E_3}\big [\kappa_g \frac{\ell}{2} X \big ]\\
&=& \ddot{\kappa}_g \eta + \dot{\kappa}_g(\kappa_g X-\frac{\ell}{2}E_3)+2 \kappa_g \dot{\kappa}_g X - \kappa_g^3 \eta \\
&& - \frac{\ell}{2} \dot{\kappa}_g E_3 - \frac{\ell^2}{4}\kappa_g \eta -\frac{\ell^2}{4}\kappa_g \eta
\end{eqnarray*}
and the conclusion follows readily. 
\end{proof}
Now, we compute:
\begin{lemma}\label{lemma-Delta2tau-cilindro}
Let $A,B,C$ be the function defined in Lemma\link\ref{lemma-nablaXX}. Then
\begin{eqnarray}\label{Delta2taucilindro}
\overline{\Delta}^2 \tau &=&\Big [\frac{\ell^2}{4}A+A \kappa_g^2-\ddot{A}-\frac{\ell}{2}B \kappa_g-2\dot{C}\kappa_g-C\dot{\kappa}_g\Big ] X \\ \nonumber
&&+ \Big [-\frac{\ell}{2}A \kappa_g +\frac{\ell^2}{4}B -\ddot{B}+\ell \dot{C} \Big ] E_3 \\ \nonumber
&&+ \Big [2 \dot{A}\kappa_g+A \dot{\kappa}_g-\ell \dot{B} -\ddot{C}+C \kappa_g^2 +\frac{\ell^2}{2}C \Big ] \eta \,.
\end{eqnarray}
\end{lemma}
\begin{proof}Obviously,
\[
\overline{\Delta}^2 \tau=\overline{\Delta}(AX)+\overline{\Delta}(BE_3)+\overline{\Delta}(C \eta)\,.
\]
Computing as in Lemma\link\ref{lemma-nablaXX} we find:
\begin{eqnarray*}
\overline{\Delta}(AX)&=&\Big [\frac{\ell^2}{4}A+A \kappa_g^2-\ddot{A}\Big]X- \Big [\frac{\ell}{2}A \kappa_g \Big ]E_3+\Big [2 \dot{A}\kappa_g+ A \dot{k}_g\Big] \eta \\
\overline{\Delta}(BE_3)&=&-\Big [\frac{\ell}{2}B \kappa_g \Big ] X+\Big [ \frac{\ell^2}{4}B-\ddot{B}\Big ]E_3 -\Big [\ell \dot{B} \Big ]\eta 
\\
\overline{\Delta}(C \eta)&=&-\Big [2\dot{C}\kappa_g +C \dot{\kappa}_g\Big]X+ \Big [\ell \dot{C} \Big ]E_3+\Big [\frac{\ell^2}{2}C+C \kappa_g^2 -\ddot{C}\Big] \eta \,.
\end{eqnarray*}
Adding up these three terms we obtain \eqref{Delta2taucilindro}.
\end{proof}
Now, using again Lemma\link\ref{lemma-nablaXX} and the explicit expression \eqref{tensore-curvatura-general-expression} of the Riemannian curvature tensor field, we can compute the two curvature terms of the $3$-tension field \eqref{2s+1-tension} and we find:
\begin{lemma}\label{lemma-R1-R2}
\begin{eqnarray*}
-\overline{R} \left(\overline{\Delta} \tau(\varphi), d \varphi (e_i)\right ) d \varphi (e_i)&=&-\frac{\ell^2}{4}A X-\frac{\ell^2}{4}B E_3+\left (-4m +\frac{\ell^2}{2} \right)C \eta \,;\\
-\,\overline{R} \Big( \nabla^\varphi_{e_i}\, \tau(\varphi),  \tau(\varphi)\Big ) d \varphi (e_i)&=&\left (4m -\frac{3\ell^2}{4} \right)\kappa_g^3 \eta\,. 
\end{eqnarray*}
\end{lemma}
Finally, adding up the terms computed in Lemmata\link\ref{lemma-Delta2tau-cilindro}--\ref{lemma-R1-R2} and simplifying using the explicit expression of the functions $A,B,C$ defined in Lemma\link\ref{lemma-nablaXX}, we obtain the explicit expression of the $3$-tension field of a Hopf cylinder. This is summarized in the following 
\begin{proposition}\label{prop-3tension-hopfcilidro} As in \eqref{Hopf-cilindro}, let $\Sigma$ be a Hopf cylinder in a BCV-space $M^3_{m,\ell}$. Then its $3$-tension field is given by
\begin{eqnarray}\label{tension3-Hopfcilindro}\nonumber
\tau_3&=&\Big [2\dot{\kappa}_g(\ell^2 \kappa_g +5 \kappa_g^3-5\ddot{ \kappa}_g )-5\kappa_g \kappa_g^{(3)} \Big ]X\\
&&+\frac{\ell}{2} \Big [-(\ell^2+9  \kappa_g^2) \dot{\kappa}_g + 4  \kappa_g^{(3)}\Big ] E_3 \\\nonumber
&& +\frac{1}{4}\Big [\kappa_g \left ( 60 \dot{\kappa}_g^2-(\ell^2+4 \kappa_g^2)(2 \ell^2-8m+\kappa_g^2) \right )+2(5 \ell^2-8m+20 \kappa_g^2) \ddot{\kappa}_g -4 \kappa_g^{(4)}\Big ] \eta \,.\\ \nonumber
\end{eqnarray}
\end{proposition}
We deduce from \eqref{tension3-Hopfcilindro} that a Hopf cylinder $\Sigma$ is $3$-harmonic if and only if the geodesic curvature $\kappa_g$ of its base curve $\gamma$ verifies:
\begin{equation}\label{harmonicity3-hopfcilingro}
\begin{cases}
2\dot{\kappa}_g(\ell^2 \kappa_g +5 \kappa_g^3-5\ddot{ \kappa}_g )-5\kappa_g \kappa_g^{(3)}=0 \\
-\ell(\ell^2+9  \kappa_g^2) \dot{\kappa}_g + 4\ell  \kappa_g^{(3)}=0 \\
\kappa_g \left ( 60 \dot{\kappa}_g^2-(\ell^2+4 \kappa_g^2)(2 \ell^2-8m+\kappa_g^2) \right )+2(5 \ell^2-8m+20 \kappa_g^2) \ddot{\kappa}_g -4 \kappa_g^{(4)}=0 \,.
\end{cases}
\end{equation}
\begin{proof}[{\bf Proof of Theorem\link\ref{Th-HopfcilindroCMC}}] Because $\tau=-\kappa_g \eta$, the proof amounts to showing that the base curve $\gamma$ of $\Sigma$ has constant geodesic curvature $\kappa_g$. We denote by $K1,K2,K3$ respectively the left-hand sides of the three equations in \eqref{harmonicity3-hopfcilingro}.

\textbf{Case $\ell \neq 0$}

We argue by contradiction. So, let us suppose that \eqref{harmonicity3-hopfcilingro} admits a nonconstant solution. It follows that there exists an open interval $I$ such that $\kappa_g(s), \dot{\kappa}_g(s)$ are both different from zero on $I$. We work on $I$. First, from 
\[
4 \ell K1+ 5 \kappa_g K2=0 
\]
we deduce that
\begin{equation}\label{Ksecondo}
\ddot{\kappa}_g= \frac{\kappa_g}{40} (3 \ell^2 -5 \kappa_g^2) \,.
\end{equation}
Next, replacing the derivative of \eqref{Ksecondo} into $K2$ gives
\[
-\frac{7}{10}\ell(\ell^2+15 \kappa_g^2) \dot\kappa_g=0  \,,
\]
Thus, since $\ell\neq 0$, we must have that  $\dot\kappa_g=0$ obtaining a contradiction.

\textbf{Case $\ell = 0$} 

In this case \eqref{harmonicity3-hopfcilingro} is equivalent to

\begin{equation}\label{harmonicity3-hopfcilingro-l=0}
\left \{
\begin{array}{ll}
{\rm (i)}&2\dot{\kappa}_g( \kappa_g^3-\ddot{ \kappa}_g )-\kappa_g \kappa_g^{(3)}=0 \\
{\rm (ii)}&15 \kappa_g\dot{\kappa}_g^2- \kappa_g^3(\kappa_g^2-8m)+2(5 \kappa_g^2-2m) \ddot{\kappa}_g -\kappa_g^{(4)}=0 \,.
\end{array}
\right .
\end{equation}
Explicit integration of \eqref{harmonicity3-hopfcilingro-l=0} (i) yields
\begin{equation}\label{eq-k2inc}
2 \kappa_g \ddot\kappa_g=\kappa^4_g-\dot\kappa^2_g +c,
\end{equation}
where $c$ is a real constant. Taking the first and the second derivative of \eqref{eq-k2inc} we obtain the expression of $\ddot \kappa_g$, $\kappa^{(3)}_g$ and $\kappa^{(4)}_g$ in terms of $\kappa_g$, $\dot\kappa_g$ and $c$. Substituting these expressions in \eqref{harmonicity3-hopfcilingro-l=0} (ii) we obtain
\begin{equation}\label{eqkprimo2-1}
c^2 + 12 m \kappa_g^6 + 7 \kappa_g^8 - 8 c \dot\kappa_g^2 + 
 7 \dot\kappa_g^4 + 4 m \kappa_g^2 (\dot\kappa_g^2-c) + 
 2 \kappa_g^4 (5 c + 7 \dot\kappa_g^2)=0\,.
\end{equation}
If $m=c=0$, then \eqref{eqkprimo2-1} becomes 
\[
7(\kappa_g^2+\dot\kappa_g^2)^2=0
\]
from which we obtain a contradiction. Now, let $(m,c)\neq(0,0)$. From \eqref{eqkprimo2-1} we obtain
\begin{equation}\label{eqkprimo2-2}
\dot\kappa_g^2=\frac{1}{7}\left(4c - 2m \kappa_g^2 - 7 \kappa_g^4\pm \sqrt{9 c^2+12 c m \kappa_g^2+(4 m^2-126 c) \kappa_g^4-56 m \kappa_g^6 }\right)\,.
\end{equation}
We can restrict our attention to an open subinterval of $I$ where 
\[9 c^2+12 c m \kappa_g^2+(4 m^2-126 c) \kappa_g^4-56 m \kappa_g^6 
\] 
is positive. Taking the derivative of \eqref{eqkprimo2-2} (we choose the solution with the sign $+$ before the square root, the other case is similar) we obtain an expression of $\ddot\kappa_g$ as a function of $\kappa_g$ and $c$. Substituting this expression and \eqref{eqkprimo2-2} into \eqref{eq-k2inc} we easily obtain that $\kappa_g$ is a root of the following polynomial with constant coefficients
\[
7056 m \kappa_g^{10} + 196 (81c+26 m^2)\kappa_g^8+384(42cm-m^3)\kappa_g^6+12c(945c-121m^2)\kappa_g^4-1764c^2m\kappa_g^2-567c^3\,.
\]
Since $(m,c)\neq(0,0)$, it follows again that $\kappa_g$ is constant, a contradiction. Thus the proof is ended.
\end{proof}
\begin{proof}[{\bf Proof of Theorem\link\ref{Th-HopfcilindroCMC-espiliciti}}] The proof follows easily from Proposition\link\ref{prop-3tension-hopfcilidro} using the assumption that $\kappa_g$ is a constant. 
\end{proof}
\section{Proof of Theorem\link\ref{Th-triharm-isop-BCV-spaces}}\label{sec-proofs2}
Our first goal is to compute all the terms which occur in the expression of the $3$-tension field
\begin{equation}\label{3-tension}
\tau_{3}(\varphi)=\overline{\Delta}^{2}\tau(\varphi)-R \left(\overline{\Delta} \tau(\varphi), d \varphi (e_i)\right ) d \varphi (e_i)\,-\,R \Big( \nabla^\varphi_{e_i}\, \tau(\varphi),  \tau(\varphi)\Big ) d \varphi (e_i)\,. 
\end{equation}
For this purpose, we now establish a series of useful preliminary lemmata.
\begin{lemma}\label{Lemma-tecnico-1} Let $\varphi:M^2 \to  M^3_{m,\ell}$ be  an oriented surface. Let $A$ denote the shape operator, $f=(1/2) \trace A$ the mean curvature function and $\eta$ the unit normal. Then
\begin{itemize}
\item[{\rm (a)}] $(\nabla A) (X,Y)=(\nabla A) (Y,X)-(R(X,Y)\eta)^\top$;
\item[{\rm (b)}] $\langle (\nabla A) (X,Y), Z \rangle=\langle (\nabla A) (X,Z), Y \rangle$;
\item[{\rm (c)}] $\trace (\nabla A) (\cdot,\cdot)= 2 \grad f+(4m-\ell^2)\, \nu \, E_3^\top$. 
\end{itemize}
\end{lemma}
\begin{proof} (a) This is just the Codazzi equation.

(b) 
\begin{eqnarray*}
\langle (\nabla A) (X,Y), Z \rangle&=&\langle \nabla_X A(Y)-A(\nabla_X Y), Z \rangle\\
&=&\langle \nabla_X A(Y),Z\rangle-\langle A(\nabla_X Y), Z \rangle\\
&=&X \langle A(Y),Z \rangle -\langle A(Y),\nabla_X Z \rangle-
\langle \nabla_X Y, A(Z) \rangle \,.
\end{eqnarray*}
Because the above expression is symmetric with respect to $Y,Z$ the conclusion follows.

(c) Let $p \in M^2$ be an arbitrarily fixed point and consider a geodesic frame field $\{ X_i\}_{i=1}^2$ around $p$. At $p$ we have: 
\begin{eqnarray*}
{\rm Trace}(\nabla A)(\cdot,\cdot)&=&\sum_{i=1}^2 (\nabla A)(X_i,X_i)=\sum_{i=1}^2 (\nabla_{X_i} A)(X_i)=\sum_{i=1}^2 \nabla_{X_i} A(X_i)\\
&=&\sum_{i,j=1}^2 \nabla_{X_i} (\langle A(X_i),X_j \rangle X_j)=\sum_{i,j=1}^2 \nabla_{X_i} (\langle X_i,A(X_j )\rangle X_j)\\
&=&\sum_{i,j=1}^2 \langle X_i, \nabla_{X_i}A(X_j) \rangle X_j=
\sum_{i,j=1}^2 \langle X_i, (\nabla A) (X_i , X_j) \rangle X_j\,.
\end{eqnarray*}
Now, first using (a) and then the explicit expression of the curvature tensor field \eqref{tensore-curvatura-general-expression}, we continue the previous sequence of equalities as follows:
\begin{eqnarray*}
{\rm Trace}(\nabla A)(\cdot,\cdot)&=&\ldots =\sum_{i,j=1}^2 \langle X_i, \nabla_{X_j} A (X_i) -\overline{R}(X_i,X_j) \eta \rangle X_j\\
&=&\sum_{i,j=1}^2 \langle X_i, \nabla_{X_j} A (X_i)\rangle X_j -\langle X_i, \overline{R}(X_i,X_j) \eta \rangle X_j \\
&=&\sum_{i,j=1}^2 \Big (X_j \langle X_i,A(X_i) \rangle \Big ) X_j\\
&&+\sum_{i,j=1}^2 (4m-\ell ^2) \Big [ \langle X_j,E_3 \rangle \nu- \langle X_i,E_3 \rangle \nu \delta_{ij}\Big ]X_j \\
&=&2 \sum_{j=1}^2 X_j ( f )X_j+(4m-\ell^2) \Big [2\nu E_3^\top-\nu E_3^\top \Big ]\\
&=&2 \grad f +(4m-\ell^2) \nu E_3^\top 
\end{eqnarray*}
and so the proof of the lemma is ended.
\end{proof}
Next, we compute:
\begin{lemma}\label{Lemma-Delta-H} Let $\varphi:M^2 \to  M^3_{m,\ell}$ be an oriented surface and denote by ${\mathbf H}= f \eta$ its mean curvature vector field. Then
\begin{equation}\label{Delta-H-formula}
 \overline{\Delta} {\mathbf H}= (\Delta f + f |A|^2) \eta + 2 A(\grad f ) + 2 f \grad f +(4m-\ell^2)\,f\, \nu \, E_3^\top \,.
\end{equation}
\end{lemma}
\begin{proof} We work with a geodesic frame field $\left \{ X_i \right \}_{i=1}^2$ around an arbitrarily fixed point $p \in M^2$. Again, we simplify the notation writing $\nabla$ for $\nabla^{M^2}$. Since ${\mathbf H}=f \eta$, around $p$ we have:
\[
\nabla_{X_i}^{\varphi} {\mathbf H}=\nabla_{X_i}^{\perp} {\mathbf H}-A_{{\mathbf H}}(X_i)=\left ( X_i f \right )\eta -f A  \left (X_i \right )\,.
\]
Then, denoting by $B$ the second fundamental form, at $p$ we have:
\begin{eqnarray*}
\nabla_{X_i}^{\varphi}\nabla_{X_i}^{\varphi} {\mathbf H}&=&\left(X_i X_i f \right )\eta- \left( X_i f \right )A  \left (X_i \right )- \left( X_i f \right )A  \left (X_i \right )-f \big( \nabla_{X_i} A  \left (X_i \right )+B\left( X_i, A  \left (X_i \right )\right )\big )\nonumber\\
&=&\left(X_i X_i f \right )\eta- 2\left( X_i f \right )A  \left (X_i \right )-f (\nabla A)(X_i,X_i)-f|A  \left (X_i \right )|^2 \eta \,.
\label{eq:nabla2H}
\end{eqnarray*}
Now, taking the sum over $i$ and using Lemma~\ref{Lemma-tecnico-1}, we obtain \eqref{Delta-H-formula} (note that the sign convention for $\Delta$ and $\overline{\Delta}$ is as in \eqref{roughlaplacian}).
\end{proof}
In the special case that the surface is CMC, say $f= \alpha$, then \eqref{Delta-H-formula} becomes:
\begin{equation}\label{Delta-H-formula-CMC}
 \overline{\Delta} {\mathbf H}= \alpha\, |A|^2 \eta + \alpha\,(4m-\ell^2)\,\nu \, E_3^\top \,.
\end{equation}
Since the mean curvature vector field and the tension field are related by $\tau(\varphi)=2 {\mathbf H}$, setting for convenience of notation $\tau(\varphi)=\tau$ and
\begin{equation}\label{V-definition}
V=\nu \, E_3^\top \,,
\end{equation}
we rewrite \eqref{Delta-H-formula-CMC} as
\begin{equation}\label{Delta-tau}
 \overline{\Delta} \tau=2 \alpha\, |A|^2 \eta +2 \alpha\,(4m-\ell^2)\,V\,.
\end{equation}
Next, from \eqref{Delta-tau} and performing a computation as in Lemma~\ref{Lemma-Delta-H} we find: 
\begin{lemma}\label{Lemma-Delta2-tau} Let $\varphi:M^2 \to  M^3_{m,\ell}$ be an oriented surface and assume that $f\equiv \alpha$. Then
\begin{equation}\label{Delta-2-H-formula}
 \overline{\Delta}^2 \tau = 2\alpha (|A|^4 + \Delta|A|^2) \, \eta + 4 \alpha A(\grad |A|^2)+2 \alpha |A|^2 (4m-\ell^2)\,V +2 \alpha  (4m-\ell^2)\,\overline{\Delta}\,V \,.
\end{equation}
\end{lemma}
The first difficulty is to compute in a convenient way $\overline{\Delta}\,V$. We have
\begin{lemma}\label{Lemma-Delta-V}

\begin{eqnarray}\label{Delta-V}
\overline{\Delta}\,V&=& 2 \nu \Big (2 A^2(E_3^\top)+\frac{3\ell}{4}\big (A(J(E_3^\top))-J(A(E_3^\top)\big) \Big )\nonumber\\ 
&&+\Big\{(4m-\ell^2)(| E_3^\top|^2-\nu^2)+|A|^2+\frac{5\ell^2}{4} \Big \}\nu E_3^\top\\\nonumber
&&+\Big\{2|A( E_3^\top)|^2+\ell\langle J  E_3^\top,A( E_3^\top) \rangle-2 \nu^2 |A|^2-(4m-\ell^2)\nu^2| E_3^\top|^2 \Big \} \eta \,.
\end{eqnarray}
\end{lemma}
\begin{proof} We can work with a geodesic frame field $\left \{ X_i \right \}_{i=1}^2$ such that $J(X_1)=X_2$, $J(X_2)=-X_1$ around an arbitrarily fixed point $p \in M^2$. Taking into account \eqref{nablaE3T} and \eqref{nablaE3-eta} we obtain around $p$:
\[
\nabla_{X_i}^\varphi V=\nu^2 \left (A(X_i)- \frac{\ell}{2} J(X_i)  \right)-
\langle A(X_i)-\frac{\ell}{2}J(X_i), E_3^\top \rangle E_3^\top +\nu \langle A(X_i), E_3^\top \rangle \eta\,. 
\]
Next, in a similar fashion, we compute at $p$:
\begin{eqnarray*}
\nabla_{X_i}^\varphi \big ( \nabla_{X_i}^\varphi V\big )&=&2 \nu \,X_i(\nu)\big ( A(X_i)- \frac{\ell}{2} J(X_i) \big ) \\
&&+\nu^2 \Big [(\nabla A) (X_i,X_i)+|A(X_i)|^2 \eta-\frac{\ell}{2}\langle A(X_i),J(X_i) \rangle \eta \Big ] \\
&&+\Big [\langle -(\nabla A)(X_i,X_i),E_3^\top \rangle -\nu \big |A(X_i)-\frac{\ell}{2}J(X_i)\big|^2 \Big ] E_3^\top \\\nonumber
&&-\langle A(X_i)- \frac{\ell}{2} J(X_i), E_3^\top \rangle\Big [\nu \big ( A(X_i)- \frac{\ell}{2} J(X_i) \big )+\langle A(X_i),E_3^\top \rangle \eta \Big ]\\
&&+X_i(\nu)\langle A(X_i),E_3^\top \rangle \eta+\nu \langle (\nabla A) (X_i,X_i),E_3^\top \rangle \eta\\
&&+\nu^2 \langle A(X_i),A(X_i)-\frac{\ell}{2}J(X_i) \rangle \eta-\nu \langle A(X_i),E_3^\top \rangle A(X_i)\,.
\end{eqnarray*}

Now, since $\left \{ X_i \right \}_{i=1}^2$ is a geodesic frame field, we have at $p$: 
\[
\overline{\Delta} V =-\sum_{i=1}^2 \nabla_{X_i}^\varphi \big ( \nabla_{X_i}^\varphi V\big ) \,.
\]
Therefore,
\begin{eqnarray}\label{DeltaV1}\nonumber
\overline{\Delta} V&=& -2 \nu \big ( A(\grad \nu)- \frac{\ell}{2} J(\grad \nu) \big ) \\\nonumber
&&-\nu^2 \Big [(4m-\ell^2)\nu E_3^\top+|A|^2 \eta-\frac{\ell}{2}\sum_{i=1}^2\langle A(X_i),J(X_i) \rangle \eta\Big ]\\
&&+\Big [(4m-\ell^2)\nu |E_3^\top|^2 +\sum_{i=1}^2 \nu \big |A(X_i)-\frac{\ell}{2}J(X_i)\big|^2 \Big ] E_3^\top \\\nonumber
&&+\sum_{i=1}^2\langle A(X_i)- \frac{\ell}{2} J(X_i), E_3^\top \rangle\Big [\nu \big ( A(X_i)- \frac{\ell}{2} J(X_i) \big )+\langle A(X_i),E_3^\top \rangle \eta \Big ]\\\nonumber
&&-\langle A(\grad \nu),E_3^\top \rangle \eta-(4m-\ell^2)\nu^2 |E_3^\top|^2 \eta\\\nonumber
&&-\nu^2\sum_{i=1}^2 \langle A(X_i),A(X_i)-\frac{\ell}{2}J(X_i) \rangle \eta+\nu \sum_{i=1}^2\langle A(X_i),E_3^\top \rangle A(X_i) \,.
\end{eqnarray} 
Now \eqref{Delta-V} can be computed using \eqref{nablaE3-eta} and performing some simplifications such as:
\[
\begin{split}
\langle A(X_1), J(X_1) \rangle+\langle A(X_2), J(X_2) \rangle&=0 \\
\sum_{i=1}^2 |A(X_i)-\frac{\ell}{2}J(X_i)\big|^2&=|A|^2 + \frac{\ell^2}{2}\\
\sum_{i=1}^2 \langle X_i,J(E_3^\top) \rangle\langle X_i,A(E_3^\top) \rangle\eta&=\langle J(E_3^\top),A(E_3^\top) \rangle \eta\\
\grad \nu=-\sum_{i=1}^2 \langle A(X_i)-\frac{\ell}{2}J(X_i),E_3^\top \rangle X_i&=-A(E_3^\top)-\frac{\ell}{2}J(E_3^\top)\,.
\end{split}
\]
Indeed,
\begin{eqnarray}\label{DeltaV2}
\overline{\Delta} V&=&2 \nu A^2(E_3^\top)+\ell \nu A(J( E_3^\top))-\ell \nu J(A( E_3^\top))+\nu \frac{\ell^2}{2}E_3^\top\nonumber \\\nonumber
&&-\nu^2 \Big [(4m-\ell^2)\nu E_3^\top+|A|^2 \eta\Big ]\\\nonumber
&&+ \Big [(4m-\ell^2)\nu |E_3^\top|^2 +\nu \left(|A|^2 +\frac{\ell^2}{2}\right) \Big ] E_3^\top \\
&&+\nu A^2 (E_3^\top) +\nu \frac{\ell}{2}A(J( E_3^\top))-\nu \frac{\ell}{2}J(A( E_3^\top))+\nu \frac{\ell^2}{4}E_3^\top\\
&&+\Big[|A(E_3^\top)|^2 +\frac{\ell}{2}\langle J(E_3^\top),A(E_3^\top) \rangle \Big ]\eta\nonumber\\\nonumber
&&+|A(E_3^\top)|^2\eta+\frac{\ell}{2}\langle J(E_3^\top),A(E_3^\top) \rangle \eta -(4m-\ell^2)\nu^2 |E_3^\top|^2 \eta\\\nonumber
&&-\nu^2 |A|^2 \eta+\nu A^2(E_3^\top )
\end{eqnarray} 
from which \eqref{Delta-V}
follows immediately (note that each line of \eqref{DeltaV1} is equal to the corresponding line of \eqref{DeltaV2}).
\end{proof}
Now, we can state the main result which is of independent interest and summarizes the preliminary work which we have carried in this section.
\begin{proposition}\label{Prop-tau3-general-expression} Let $M^2$ be an oriented surface in $M^3_{m,\ell}$. Assume that $M^2$ has CMC equal to $\alpha$. Then its $3$-tension field is
\begin{eqnarray}\label{tau3-general-case}
\nonumber
\tau_3&=&  \alpha \Big\{2\Delta|A|^2+ 2|A|^4 + |A|^2 \big[\ell^2(1+2\nu^2)-8m (1+\nu^2) \big]\\\nonumber
    && +4(4m-\ell^2)|A(E_3^\top)|^2+2\ell(4m-\ell^2) \langle J(E_3^\top),A(E_3^\top)\rangle -4 \alpha(4 m - \ell^2) \langle A(E_3^\top),E_3^\top\rangle\\
    && +2 \alpha^2 \ell^2 (3+4\nu^2)-32 m \alpha^2 (1+\nu^2)\Big \} \eta\\\nonumber
      &&+2 \alpha(4 m-\ell^2) \nu \Big \{3 |A|^2 + 4 \alpha^2 + 3 \ell^2 \nu^2 + 
   m (4 - 12 \nu^2)\Big \} E_3^\top \\\nonumber
   &&+4 \alpha(4m-\ell^2)\nu \Big \{2 A^2(E_3^\top)-\frac{3\ell}{4}J(A(E_3^\top))+\frac{3\ell}{4}(A(J(E_3^\top)) \Big \}-4 \alpha^2(4 m - \ell^2)\nu  A(E_3^\top)\\\nonumber
   &&+4\alpha A(\grad |A|^2)\,,
\end{eqnarray}
where $\nu$ is defined in \eqref{nu-definition}.
\end{proposition}
\begin{proof}
The explicit expression of the curvature tensor field is given in \eqref{tensore-curvatura-general-expression} and so we have all the ingredients to compute the $3$-tension field \eqref{3-tension}. More in detail:

\textbf{(I)} $\overline{\Delta}^2 \tau$ is given in Lemmata\link\ref{Lemma-Delta2-tau} and \ref{Lemma-Delta-V}.

\textbf{(II)} Here we compute the first curvature term
\[
\overline{R} \left(\overline{\Delta} \tau(\varphi), d \varphi (e_i)\right ) d \varphi (e_i) \,,
\]
which we rewrite as 
\begin{equation}\label{tau3-pezzo-II}
\sum_{i=1}^2 \overline{R} \left((\overline{\Delta} \tau)^\perp, X_i\right ) X_i+\sum_{i=1}^2 \overline{R} \left((\overline{\Delta} \tau)^\top, X_i\right ) X_i \,,
\end{equation}
where, according to \eqref{Delta-tau}, we have
\begin{equation}\label{Delta-tau-perp-top}
(\overline{\Delta} \tau)^\perp=2 \alpha\, |A|^2 \eta\,, \quad  (\overline{\Delta} \tau)^\top=2 \alpha\,(4m-\ell^2)\,V \,.
\end{equation}
Then, using the general expression \eqref{tensore-curvatura-general-expression} for the curvature tensor field, we find:
\begin{eqnarray}\label{pezzo-Delta-tau-perp}
\sum_{i=1}^2 \overline{R} \left((\overline{\Delta} \tau)^\perp, X_i\right ) X_i&=&2 \alpha |A|^2 \Big \{ 2 \left ( 4m -\frac{3 \ell^2}{4}\right )\eta \\\nonumber
&& \qquad \qquad -(4 m - \ell^2) \big [(|E_3^\top|^2 + 2 \nu^2)\eta + \nu E_3^\top \big ] \Big \} \,.
\end{eqnarray}
\begin{eqnarray}\label{pezzo-Delta-tau-top}
\sum_{i=1}^2 \overline{R} \left((\overline{\Delta} \tau)^\top, X_i\right ) X_i&=&2 \alpha (4 m - \ell^2)\nu  \Big \{ \left ( 4m -\frac{3 \ell^2}{4}\right )\,E_3^\top\\\nonumber
&& \qquad \qquad \qquad  -(4 m - \ell^2) \big [|E_3^\top|^2 E_3^\top + |E_3^\top|^2  \nu \, \eta  \big ] \Big \} \,.
\end{eqnarray}
Adding up the results in \eqref{pezzo-Delta-tau-perp} and \eqref{pezzo-Delta-tau-top} we can easily handle the term in \textbf{(II)}.

\textbf{(III)} Here we deal with the other curvature term, i.e.,
\[
\overline{R} \Big( \nabla^\varphi_{e_i}\, \tau(\varphi),  \tau(\varphi)\Big ) d \varphi (e_i)\,.
\]
Using again the general expression \eqref{tensore-curvatura-general-expression} for the curvature tensor field and observing that
\[
\nabla^\varphi_{X_i} \tau=- 2 \alpha A (X_i)
\]
we find (slight abuse of notation: we identify $X_i$ and $d \varphi(X_i)$):
\begin{eqnarray}\label{pezzo-nabla-tau}\nonumber
\sum_{i=1}^2 \overline{R} \left(\nabla_{X_i}^\varphi \tau, \tau \right ) X_i&=& -4\alpha^2 \sum_{i=1}^2 \overline{R} \left(A(X_i), \eta \right ) X_i\\
&=&-4 \alpha^2\Big \{-2 \alpha  \left ( 4m -\frac{3 \ell^2}{4}\right )\eta \\
&& +(4 m - \ell^2) \Big [\nu A(E_3^\top)-\langle A(E_3^\top),E_3^\top\rangle \eta -2 \alpha \nu E_3^\top-2 \alpha \nu^2 \eta\Big ] \Big \}\nonumber
\end{eqnarray}
Using the results obtained in the three cases \textbf{(I)}, \textbf{(II)}, \textbf{(III)} we can easily compute the $3$-tension field described explicitly in \eqref{3-tension}. More precisely, after some routine simplifications we obtain \eqref{tau3-general-case} (where we have used that $|E_3^\top|^2=1-\nu^2$) and so the proof of Proposition\link\ref{Prop-tau3-general-expression} is completed.
\end{proof}
\begin{proof}[{\bf Proof of Theorem\link\ref{Th-triharm-isop-BCV-spaces}}] According to Theorem\link\ref{Th-isop-BCV-spaces}, we just have to study the three possible cases (iv)(a),(b),(c). Case (iv)(a) is analysed in detail in Theorem\link\ref{Th-HopfcilindroCMC-espiliciti} and provides examples of proper triharmonic surfaces. By contrast, horizontal slices of the type (iv)(b) are totally geodesic and therefore there exists no proper triharmonic surface of this type. By way of summary, the proof will be complete if we show that any parabolic helicoid $P_{\alpha,m,\ell}$ as in Theorem\link\ref{Th-isop-BCV-spaces} (iv)(c) cannot be proper triharmonic. Despite the simplicity of this plan for the proof, the involved computations are quite long and will be carried out using the \textit{half-space model} for $M^3_{m,\ell}$. More precisely, recalling that in the case of parabolic helicoids $m<0$ by assumption, we shall work in
$$
\tilde{M}^3_{m,\ell}=\left ( \{(\tilde x,\tilde y,\tilde z)\in \R^3 \,: \, \tilde y>0\},\,\frac{d\tilde x^2+d\tilde y^2}{-4m \tilde y^2} +\left ( d\tilde z+\frac{\ell}{4m\tilde y}d\tilde x \right )^2\right )\,.
$$
Now, we recall (see \cite{manzano-torralbo}) some basic facts about the half-space model. We have an explicit isometry $\Theta:M^3_{m,\ell}\to\tilde{M}^3_{m,\ell}$ given by
\[
\Theta (x,y,z)=\left (\frac{2y}{\sqrt{-m}G^2},\frac{1+m(x^2+y^2)}{-mG^2},z+\frac{\ell}{2m}\arccos \left( \frac{y}{G}\right) \right )\,,
\]
where we have set
\[
G=\sqrt{\left ( \frac{1}{\sqrt{-m}}+x\right )^2+y^2} \,.
\]
Moreover, an explicit positively oriented global orthonormal frame field on the half-space model is:
\[
\tilde{E}_1= 2 \sqrt{-m}\,\tilde y \,\partial_{\tilde x}+\frac{\ell}{2 \sqrt{-m}}\,\partial_{\tilde z}\,, \quad \tilde{E}_2=2 \sqrt{-m}\,\tilde y \,\partial_{\tilde y}\,,\quad \tilde{E}_3= \partial_{\tilde z} \,.
\]
For future use, we observe that $d \Theta(E_3)=\tilde{E}_3$. Next, we compute
\[
[\tilde{E}_1,\tilde{E}_2]=-2\sqrt{-m}\tilde{E}_1+\ell \tilde{E}_3 \,; \quad [\tilde{E}_1,\tilde{E}_3]=[\tilde{E}_2,\tilde{E}_3]=0\,.
\]
Then, using the Koszul identity \eqref{Koszul}, it is easy to verify that the version of \eqref{CNil} in this context is:
\begin{equation}\label{CNil-half-space}
\begin{array}{lll}
\overline{\nabla}_{\tilde{E}_1}\tilde{E}_1=2 \sqrt{-m}\tilde{E}_2\,, & \overline{\nabla}_{\tilde{E}_1}\tilde{E}_2=-2 \sqrt{-m}\tilde{E}_1+\dfrac{\ell}{2}\tilde{E}_3\,, &\overline{\nabla}_{\tilde{E}_1}\tilde{E}_3=-\dfrac{\ell}{2}\tilde{E}_2\,, \vspace{2mm}\\
\overline{\nabla}_{\tilde{E}_2}\tilde{E}_1=-\dfrac{\ell}{2}\tilde{E}_3\,,  & \overline{\nabla}_{\tilde{E}_2}\tilde{E}_2= 0&\overline{\nabla}_{\tilde{E}_2}\tilde{E}_3=\dfrac{\ell}{2}\tilde{E}_1\,, \vspace{2mm} \\
\overline{\nabla}_{\tilde{E}_3}\tilde{E}_1=-\dfrac{\ell}{2}\tilde{E}_2\,,  & \overline{\nabla}_{\tilde{E}_3}\tilde{E}_2=\dfrac{\ell}{2}\tilde{E}_1\,, &\overline{\nabla}_{\tilde{E}_3}\tilde{E}_3= 0\,.
\end{array} 
\end{equation}
The parabolic helicoids $P_{\alpha,m,\ell}$ (see \cite{MR4288655}) are the CMC surfaces in $\tilde{M}^3_{m,\ell}$ parametrized by
\[
X(u,v)=\left(u,v,a \log v \right ) \,,\quad v>0 \,,
\]
where $a$ is non-vanishing real constant whose relation with the geometrical parameters $\alpha,m,\ell$ will be made explicit in \eqref{a-in-terms-of-alpha} below.

We have to verify that a parabolic helicoid $P_{\alpha,m,\ell}$, $\alpha\neq 0$, cannot be triharmonic. For this purpose, we apply Proposition\link\ref{Prop-tau3-general-expression}. In order to compute all the terms which appear in the expression of the $3$-tension field $\tau_3$ (see \eqref{tau3-general-case} it is convenient to express and compute all the relevant quantities with respect to the global orthonormal frame field $\{\tilde{E}_1,\tilde{E}_2,\tilde{E}_3 \}$. Writing $[\tilde{x}_1,\tilde{x}_2,\tilde{x}_3]$ for $\tilde{x}_1\tilde{E}_1+\tilde{x}_2\tilde{E}_2+\tilde{x}_3\tilde{E}_3$, we compute:
\begin{eqnarray}\label{XuXveta}
X_u&=&\left [\frac{1}{2 \sqrt{-m}\,v},0,\frac{\ell}{4mv} \right ] \\\nonumber
X_v&=&\left [0,\frac{1}{2 \sqrt{-m}\,v},\frac{a}{v} \right ] \\\nonumber
\eta&=&\frac{1}{L} \Big [\ell,4am,2\sqrt{-m}\Big ]\,,\nonumber
\end{eqnarray}
where we have set
\begin{equation}\label{eq:defofL}
L=\sqrt{\ell^2 + 4 m (4 a^2 m-1)}
\end{equation}
Note that 
$$
\nu=\eta \cdot \tilde{E}_3=\frac{2\sqrt{-m}}{L}
$$ 
and so we can compute $E_3^\top=\tilde{E}_3-\nu \eta$. We find
\[
E_3^\top=\left [-\frac{2 \ell \sqrt{-m}}{L^2},\frac{8 a (-m)^{3/2}}{L^2},\frac{4 m}{L^2}+1 \right ] \,.
\]
Moreover,
\[
J(E_3^\top)=\eta \times E_3^\top=\left [\frac{4am}{L},-\frac{\ell}{L},0 \right ] \,.
\]
Now, we recall that $A(X)=-\overline{\nabla}_X \eta$ and $J(X)=\eta \times X$ for all $X$ tangent to $P_{\alpha,m,\ell}$. Next, we observe that all the coefficients of $\eta,E_3^\top,J(E_3^\top)$ are constant. Therefore, using \eqref{CNil-half-space} it is not difficult to compute:

\begin{equation}\label{pezzi-per-tau3}
\begin{aligned}
A(E_3^\top)&=\left [-\frac{2 a \ell m}{L},\frac{\ell^2}{2 L},0 \right ]\\
A(J(E_3^\top))&=\left [ \frac{\sqrt{-m} \left(32 a^2 m^2+\ell^2\right)}{L^2},\frac{4 a \ell
   (-m)^{3/2}}{L^2},-\frac{16 a^2 \ell m^2+\ell^3}{2 L^2}\right ]\\
   A^2(E_3^\top)&=\left [-\frac{\ell^2 \sqrt{-m}}{L^2},\frac{4 a \ell (-m)^{3/2}}{L^2},\frac{16
   a^2 \ell m^2+\ell^3}{2 L^2} \right]\\
  J(A(E_3^\top)) &=\left [ -\frac{\ell \sqrt{-m} \left(32 a^2 m^2+\ell^2\right)}{2 L^2},-\frac{2 a
   \ell^2 (-m)^{3/2}}{L^2},\frac{16 a^2 \ell^2 m^2+\ell^4}{4 L^2}\right]
   \end{aligned}
\end{equation}
Using \eqref{pezzi-per-tau3}, we are in the right position to compute the mean curvature $\alpha$ and $|A|^2$. We have
\begin{equation}\label{alpha-in-terms-of-a}
2\alpha=\frac{1}{|E_3^\top|^2}\big(\langle A(E_3^\top), E_3^\top\rangle+\langle A(J(E_3^\top)), J(E_3^\top)\rangle\big)=\frac{8 a (-m)^{3/2}}{L}. 
\end{equation}
Inverting \eqref{alpha-in-terms-of-a}, together with \eqref{eq:defofL}, we deduce that the relationship between $a$ and the geometrical parameters $\alpha, \ell,m$ is:
\begin{equation}\label{a-in-terms-of-alpha}
a=\frac{\alpha}{4m} \sqrt{\frac{4m- \ell^2}{m+\alpha^2}}.
\end{equation}
As for $|A|^2$, using \eqref{a-in-terms-of-alpha} and \eqref{eq:defofL} a straight computation returns
\begin{equation}\label{valueofA2}
|A|^2=\frac{1}{|E_3^\top|^2}\big(\langle A(E_3^\top), A(E_3^\top)\rangle+\langle A(J(E_3^\top)), A(J(E_3^\top))\rangle\big)=\frac{1}{2}(8\alpha^2 + \ell^2).
\end{equation}
Next, using \eqref{pezzi-per-tau3}, \eqref{a-in-terms-of-alpha} and  \eqref{valueofA2} all the terms in \eqref{tau3-general-case} can be easily computed and, after adding up and simplifying, we find
\[
\tau_3=
  \alpha\,\left[\frac{1}{\sqrt{-m}}\sqrt{\frac{\alpha ^2+m}{4m-\ell^2}}\,T1,\frac{\alpha}{\sqrt{-m}}\,T2,-2\sqrt{\frac{\alpha ^2+m}{4m-\ell^2}}\,T3 \right ] \,,
\]
where
\begin{eqnarray*}
T1&=&\ell \Big [ \ell^4 + \ell^2 (-16 m +6 \alpha^2) + 
 32  (2 m^2 - 3 m\alpha^2 -2 \alpha^4)\Big]\,;\\
T2&=& \ell^4 + \ell^2 (-16 m +6 \alpha^2) + 
 32  (2 m^2 - 3 m\alpha^2 -2 \alpha^4)\,;\\
T3&=&\ell^4  + \ell^2 (-8 m + 2 \alpha^2)+ 64 \,\alpha^2 (m + \alpha^2)\,.
\end{eqnarray*}
By way of summary, a parabolic helicoid $P_{\alpha,m,\ell}$ is proper triharmonic if and only if
\begin{equation}\label{systemT1T2T3=0}
\left \{ 
\begin{array}{l}
T1=0 \\
T2=0 \\
T3=0 \,.
\end{array}
\right .
\end{equation}
Now, if $\ell=0$, the third equation of system \eqref{systemT1T2T3=0} becomes
\[
64\,\alpha^2 (m+\alpha^2)=0 \,,
\]
which has not relevant solutions since in our construction $m+\alpha^2<0$ by hypothesis.

Next, we handle the case $\ell\neq 0$. First, we observe that the condition $T2-T3=0$ implies
\begin{equation}\label{ell^2}
\ell^2= \frac{-8(-2m^2+5m \alpha^2 +4\alpha^4)}{2m-\alpha^2}\,.
\end{equation}
Now, replacing this value of $\ell^2$ into $T2$ and simplifying we find that necessarily
\[
(2m+3\alpha^2) (-2m^2+9m\alpha^2+8\alpha^4)=0\,.
\]
Because $m<0$, the only possibilities are:
\begin{equation}\label{m}
m=-\frac{3 \alpha^2}{2} \quad {\rm or}\quad m= \frac{\alpha^2}{4}(9-\sqrt{145})\,.
\end{equation}
The first value for $m$ is not acceptable because, if we replace it into \eqref{ell^2}, we find
\[
\ell^2=-16 \alpha^2 \,,
\]
a contradiction. As for the second value of $m$ into \eqref{m}, it suffices to observe that it would imply $m+\alpha^2 >0$, a fact which contradicts our assumption.

Therefore, there exists no proper triharmonic parabolic helicoid $P_{\alpha,m,\ell}$ and so the proof of  Theorem\link\ref{Th-triharm-isop-BCV-spaces} is completed.
\end{proof}
\section{Proper CMC $r$-harmonic Hopf cylinders}\label{sec-r-harmonic}
In this section we focus on the study of proper CMC $r$-harmonic Hopf cylinders. We shall show that the existence of such submanifolds depends not only on the curvature of the ambient space, but also on the value of $r$. This will be illustrated in Corollary~\ref{Cor-r-harmonic-hopf-cilindri} which is a consequence of the following theorem.
\begin{theorem}\label{Th-r-harmonic-hopf-cilindri} Assume that $r \ge2 $. Let $\Sigma$ be a non minimal CMC Hopf cylinder in a BCV-space $M^3_{m,\ell}$. Then $\Sigma$ is proper $r$-harmonic if and only if the geodesic curvature $\kappa_g$ of its base curve $\gamma$ is a non zero constant which verifies
\begin{equation}\label{cond-r-harmonicity-cilindri}
\kappa_g^4  +  \left [-4 m ( r-1) + \frac{3 \ell^2 r}{4}\right ]\kappa_g^2- \frac{\ell^2}{2} (4 m-\ell^2)=0\,.
\end{equation}
\end{theorem}
\begin{remark} From \eqref{lemma-nablaXX} (ii) and (iv), the norm $|A|^2$ of the shape operator of $\Sigma$ is (see also \cite{Ou-J-Geom-Phys-2011}):
\begin{equation}\label{q1}
|A|^2=\kappa_g^2+\frac{\ell^2}{2}\,.
\end{equation}
Using \eqref{q1}, condition \eqref{cond-r-harmonicity-cilindri} is equivalent to
\begin{equation}\label{cond-r-harmonicity-cilindri-bis}
|A|^4-\left (4m-\frac{\ell^2}{2} \right )|A|^2-(r-2)\left (4m-\frac{3\ell^2}{4} \right )\kappa_g^2=0\,.
\end{equation}
Thus, setting $r=2$ into \eqref{cond-r-harmonicity-cilindri-bis}, it is immediate to recover the result of Ou \cite{Ou-J-Geom-Phys-2011} concerning CMC biharmonic Hopf cylinders.
Also, in the special case $r=3$ in \eqref{cond-r-harmonicity-cilindri}, it is easy to recover the statement of Theorem\link\ref{Th-HopfcilindroCMC-espiliciti}.
\end{remark}
\vspace{2mm}

In the following corollary we shall indicate for which values of the parameters $\ell$, $m$ and $r$ there are acceptable solutions  of \eqref{cond-r-harmonicity-cilindri}. We suggest that the reader keeps in mind the geometrical counterpart of the cases (i), (ii), (iii) of Corollary\link\ref{Cor-r-harmonic-hopf-cilindri} referring to the diagram in Figure~\ref{diagram-BCVspaces}.

\begin{corollary}\label{Cor-r-harmonic-hopf-cilindri}
Let $M^3_{m,\ell}$  be a BCV-space. Then there exists a proper CMC $r$-harmonic Hopf cylinder if and only if one of the following holds:
\begin{enumerate} 
\item[{\rm (i)}] $4m-\ell^2>0$ and $r\geq 2$;
\item[{\rm (ii)}] $4m-\ell^2=0$, $\ell\neq 0$ and $r\geq 5$;
\item[{\rm (iii)}] $4m-\ell^2<0$, $\overline{R}_{1212}=4m-{3 \ell^2}/{4}>0$ and 
\begin{equation}\label{limit-value-for-r}
r\geq \frac{4 \left(\sqrt{2} \sqrt{\ell^4-4 \ell^2 m}+4
   m\right)}{16 m -3 \ell^2}\,.
\end{equation}
\end{enumerate}
\end{corollary}

\begin{remark}
The assumptions  $4m-\ell^2<0$ and $\overline{R}_{1212}=4m-{3 \ell^2}/{4}>0$ in Case (iii) of Corollary\link\ref{Cor-r-harmonic-hopf-cilindri} have a geometrical meaning because they state that the ambient space is $SU(2)$ endowed with a metric with positive sectional curvature. Here we point out that from the analytical view point these hypotheses are equivalent just to the condition $(3/4)\ell^2 < 4m < \ell^2$ which corresponds to the region between the two parabolas $4m = (3/4) \ell^2$ and $4m=\ell^2$ (see Figure~\ref{diagram-BCVspaces_hopf-cylinders1}). It is convenient to describe this region as the union $\cup_{a} \gamma_a,\, 3/4<a<1$, where $\gamma_a$ is the parabola $4m=a\ell^2$. Now, on $\gamma_a$ the lower bound for $r$  in \eqref{limit-value-for-r} becomes 
\[
r_a=\frac{4 \left(a+\sqrt{2-2 a}\right)}{4 a-3}\,.
\]
We observe that on $(3/4,1)$ $r_a$ is a strictly decreasing function of $a$ with $\lim_{a\to3/4^+}r_a=+\infty$ and $\lim_{a\to 1^-}r_a=4$. Therefore, the more we approach $\gamma_{3/4}$, the larger $r$ must be in order to have a proper CMC $r$-harmonic Hopf cylinder. 
\end{remark}

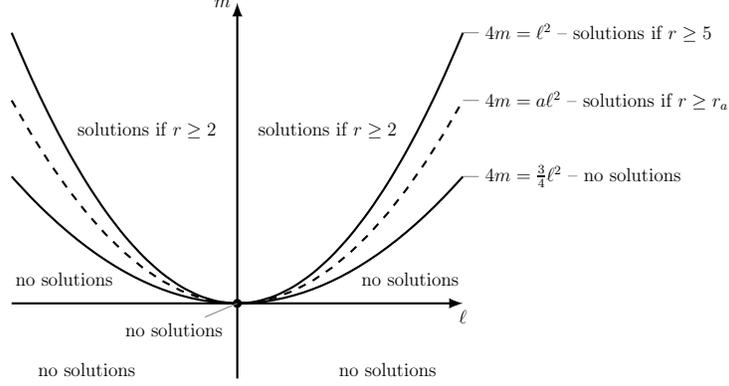
\begin{figure}[h]
\caption{Distribution of proper CMC $r$-harmonic Hopf cylinders in BCV-spaces w.r.t. to the values of $\ell$, $m$ and $r$.} 
\label{diagram-BCVspaces_hopf-cylinders1}

\begin{tikzpicture}
\draw[thick, -latex] (-3,0) -- (3,0) node[below,scale=0.6] {$\ell$};
\draw[thick, -latex] (0,-1) -- (0,4) node[left,scale=0.6] {$m$};

\draw[scale=1, domain=-3:3, smooth, variable=\x,  thick] plot ({\x}, {0.4*\x*\x});
\coordinate[pin={[pin distance=10,scale=0.6]0:$4m=\ell^2$ -- solutions if $r\geq 5$}] (r1) at (3,0.4*3*3);

\draw[scale=1, domain=-3:3, smooth, variable=\x,  thick] plot ({\x}, {0.1875*\x*\x});
\coordinate[pin={[pin distance=10,scale=0.6]0:$4m=\frac{3}{4}\ell^2$ -- no solutions}] (r1) at (3,0.1875*3*3);

\draw[scale=1, domain=-3:3, dashed, variable=\x,  thick] plot ({\x}, {0.3*\x*\x});
\coordinate[pin={[pin distance=10,scale=0.6]0:$4m=a \ell^2$ -- solutions if $r\geq r_a$}] (r1) at (3,0.3*3*3);

\node[below,scale=0.6] at (1.2,2.5) {solutions if $r\geq 2$};
\node[below,scale=0.6] at (-1.2,2.5) {solutions if $r\geq 2$};
\node[below,scale=0.6] at (2.3,0.5) {no solutions};
\node[below,scale=0.6] at (-2.3,0.5) {no solutions};
\node[below,scale=0.6] at (2.,-.7) {no solutions};
\node[below,scale=0.6] at (-2.,-.7) {no solutions};
\filldraw [black] (0,0) circle (1.5pt);
 \coordinate[pin={[pin distance=10,scale=0.6]240:no solutions}] (r1) at (0,0);
\end{tikzpicture}
\end{figure}

Now, we prove the results of this section.

\begin{proof}[\bf{Proof of Theorem\link\ref{Th-r-harmonic-hopf-cilindri}}]We know that $\tau=-\kappa_g \eta$ and $\kappa_g$ is constant. To simplify the notation, we set
\[
c=\kappa_g^2 + \frac{\ell^2}{2} \quad (=|A|^2) \,.
\]
We know from Lemma\link\ref{lemma-Deltatau-cilindro} that $\overline{\Delta}\tau=-\kappa_g c \eta$ and then we deduce that\begin{equation}\label{Deltaktau}
\overline{\Delta}^r\tau=-\kappa_g c^r \eta \,.
\end{equation}
Next, using Lemma\link\ref{lemma-nablaXX}, we obtain:
\begin{equation}\label{nablaDeltaktau}
\overline{\nabla}_X (\overline{\Delta}^r\tau)=-\kappa_g^2 c^r X+\frac{\ell}{2}\kappa_g c^r E_3\,; \qquad \overline{\nabla}_{E_3} (\overline{\Delta}^r\tau)=\frac{\ell}{2}\kappa_g c^r X\,. 
\end{equation}
We shall also need the following equalities which can easily by derived from \eqref{tensore-curvatura-general-expression}:
\begin{equation}\label{terminidicurvaturaespliciti}
\overline{R}(\eta,X)X=\left(4m-\frac{3 \ell^2}{4} \right ) \eta\, ; \quad 
\overline{R}(\eta,E_3)E_3=\frac{\ell^2}{4}  \eta\, ; \quad \overline{R}(X,\eta)E_3 =0 \,.
\end{equation}
Now, using \eqref{Deltaktau}, \eqref{nablaDeltaktau} and \eqref{terminidicurvaturaespliciti}, we can compute the $r$-tension field whose expression is given in \eqref{2s-tension}, \eqref{2s+1-tension}. We have:
\begin{eqnarray*}\label{2s-tension-hopfcilindro}
\tau_{2s}(\varphi)&=&\overline{\Delta}^{2s-1}\tau(\varphi)-R^N \left(\overline{\Delta}^{2s-2} \tau(\varphi), d \varphi (e_i)\right ) d \varphi (e_i) \nonumber\\ 
&&  - \sum_{p=1}^{s-1}\, \left \{R^N \left( \nabla^\varphi_{e_i}\,\overline{\Delta}^{s+p-2} \tau(\varphi), \overline{\Delta}^{s-p-1} \tau(\varphi)\right ) d \varphi (e_i)  \right .\\ \nonumber
&& \qquad \qquad  -\, \left . R^N \left( \overline{\Delta}^{s+p-2} \tau(\varphi),\nabla^\varphi_{e_i}\, \overline{\Delta}^{s-p-1} \tau(\varphi)\right ) d \varphi (e_i)  \right \} \\
&=&-\kappa_g c^{2s-1}\eta+\kappa_g c^{2s-2}\Big [\overline{R}(\eta,X)X+ \overline{R}(\eta,E_3)E_3\Big ]\\
&&- \sum_{p=1}^{s-1}\, \Big \{\kappa_g^3 c^{2s-3}\overline{R}(X,\eta)X-\frac{\ell}{2}\kappa_g^2 c^{2s-3}\overline{R}(E_3,\eta)X\\
&&\qquad\quad  \quad -\frac{\ell}{2}\kappa_g^2 c^{2s-3}\overline{R}(X,\eta)E_3-\kappa_g^3 c^{2s-3}\overline{R}(\eta,X)X \\
&&\qquad \quad \quad+\frac{\ell}{2}\kappa_g^2 c^{2s-3}\overline{R}(\eta,E_3)X-\frac{\ell}{2}\kappa_g^2 c^{2s-3}\overline{R}(\eta,X)E_3 \Big \}\\
&=&-\kappa_g c^{2s-1}\eta+\kappa_g c^{2s-2}\Big [\left(4m-\frac{3 \ell^2}{4} \right ) \eta+\frac{\ell^2}{4}  \eta\Big ]\\
&&-c^{2s-3} \sum_{p=1}^{s-1}\, \Big \{-\kappa_g^3\left(4m-\frac{3 \ell^2}{4} \right ) \eta-\kappa_g^3\left(4m-\frac{3 \ell^2}{4} \right ) \eta\Big \}\\
&=&-\kappa_g c^{2s-3} \Big [c^2- \left(4m-\frac{ \ell^2}{2} \right ) c-(2s-2)\left(4m-\frac{3 \ell^2}{4} \right ) \kappa_g^2\Big ] \eta\,.
\end{eqnarray*}
It follows from the last equality that the Hopf cylinder $\Sigma$ is proper $r$-harmonic ($r=2s$) if and only if
\begin{equation}\label{cond-r-harmonicity-cilindri-tris}
c^2-\left (4m-\frac{\ell^2}{2} \right )c-(r-2)\left (4m-\frac{3\ell^2}{4} \right )\kappa_g^2=0\,.
\end{equation}
Finally, using $c=\kappa_g^2+(\ell^2\slash2)$ in \eqref{cond-r-harmonicity-cilindri-tris} we obtain \eqref{cond-r-harmonicity-cilindri}.

The computations for $\tau_{2s+1}$ are analogous and so we omit further details.
\end{proof}
\begin{proof}[{\bf Proof of Corollary\link\ref{Cor-r-harmonic-hopf-cilindri}}] We will proceed by an accurate  analysis of the dependence of the roots of equation \eqref{cond-r-harmonicity-cilindri} on the values of the parameters $\ell$, $m$ and $r$.
Putting $x=k_g^2$ in \eqref{cond-r-harmonicity-cilindri}, the existence of a proper $r$-harmonic CMC Hopf cylinder is equivalent to the existence of a positive solution of the equation
\begin{equation}\label{cond-r-harmonicity-cilindri-x}
x^2  +  \left [-4 m ( r-1) + \frac{3 \ell^2 r}{4}\right ]x- \frac{\ell^2}{2} (4 m-\ell^2)=0\,.
\end{equation}
We divide the analysis in a series of cases.

If $m\leq 0$, then the coefficients of \eqref{cond-r-harmonicity-cilindri-x} are nonnegative and thus there exists no positive solution.

Thus, from now on, we assume that $m>0$. Then we have the following subcases:

\begin{itemize}
\item[(i)] If $4m -\ell^2 > 0$, then there exists a unique positive solution $\kappa_g^2$ of \eqref{cond-r-harmonicity-cilindri-x} for all $r\geq2$.
\item[(ii)] If $4m -\ell^2 = 0$, then equation \eqref{cond-r-harmonicity-cilindri-x} admits the positive solution
\[
\kappa_g^2= \frac{r-4}{4}\,
\] 
if and only if $r \geq5$. We point out that this result is in accordance with the discussion in  Remark~1.3 of \cite{MR3711937}.
\item[(iii)] If $4m -\ell^2 < 0$, then replacing $\overline{R}_{1212}=4m-{3 \ell^2}/{4}$ in \eqref{cond-r-harmonicity-cilindri-x} we obtain
\begin{equation}\label{cond-r-harmonicity-cilindri-x-R}
x^2  +  \left [4 m -r \overline{R}_{1212}\right]x- \frac{\ell^2}{2}\overline{R}_{1212}+\frac{\ell^4}{8}=0\,.
\end{equation}
Thus, if $\overline{R}_{1212}\leq 0$ there exists no acceptable solution. Finally,  if $\overline{R}_{1212}=4m-{3 \ell^2}/{4}>0$, then a straightforward check shows that \eqref{cond-r-harmonicity-cilindri-x} admits two distinct positive solutions if and only if
\begin{equation}\label{r-2-cilindri}
r> \frac{4 \left(\sqrt{2} \sqrt{\ell^4-4 \ell^2 m}+4
   m\right)}{16 m -3 \ell^2} \,,
\end{equation}
and only one positive solution when 
\begin{equation}\label{r-1-cilindro}
r=\frac{4 \left(\sqrt{2} \sqrt{\ell^4-4 \ell^2 m}+4
   m\right)}{16 m -3 \ell^2}\,.
\end{equation}
\end{itemize}
\end{proof}
\begin{remark}
(i) We observe that, for all $r \geq 5$, there always exist suitable couples $m,\ell$ such that \eqref{r-2-cilindri}, or \eqref{r-1-cilindro}, is verified. In the case that \eqref{r-2-cilindri} holds the corresponding two solutions give rise to two \textit{non-congruent} $r$-harmonic Hopf cylinders.

(ii) Concerning the case $4m -\ell^2 > 0$, if $r=2$ we find $\kappa_g^2=4m-\ell^2$. On the other hand, Ou and Wang (see Theorem\link\ref{CVV}~(iii)) find that $\gamma$ must be a circle of radius
\begin{equation}\label{eq:ou-wang}
R=\frac{1}{\sqrt{8m-\ell^2}}
\end{equation}
in the base space $\s^2\left (\frac{1}{2 \sqrt m} \right)=M^2(4m)$. In order to check that our result is coherent with that of Ou and Wang, we observe that a curve of constant curvature $\kappa_g$ in $\s^2(\rho)$ is a plane curve in $\R^3$  with constant curvature $\kappa$ and radius
\begin{equation}\label{eq:link-kg-k}
R=\frac{1}{\kappa}=\frac{\rho}{\sqrt{\kappa_g^2 \rho^2+1}}\,.
\end{equation}
Substituting $\rho=1/(2\sqrt{m})$ and $\kappa_g^2=4m-\ell^2$ in \eqref{eq:link-kg-k} we recover \eqref{eq:ou-wang}.
\end{remark}


\end{document}